\documentclass[a4paper]{article}
\usepackage{amssymb}
\usepackage{amsmath}
\usepackage{amsthm}
\usepackage{mathtools}
\usepackage{hyperref}
\usepackage{enumerate}
\usepackage{graphicx}

\newtheorem{Theorem} {Theorem} [section]
\newtheorem{Proposition} [Theorem] {Proposition}
\newtheorem{Lemma} [Theorem] {Lemma}
\newtheorem{Corollary} [Theorem] {Corollary}
\newtheorem{Example} [Theorem] {Example}
\newtheorem{Problem} [Theorem] {Problem}
\newtheorem{Remark} [Theorem] {Remark}
\newtheorem{Conjecture} [Theorem] {Conjecture}

\newcommand{\Ef}{{\mathbb E}}
\newcommand{\Ff}{{\mathbb F}}
\newcommand{\Rf}{{\mathbb R}}

\newcommand{\cC}{{\mathcal C}}
\newcommand{\cD}{{\mathcal D}}
\newcommand{\cF}{{\mathcal F}}
\newcommand{\cH}{{\mathcal H}}
\newcommand{\cL}{{\mathcal L}}
\newcommand{\cP}{{\mathcal P}}
\newcommand{\cQ}{{\mathcal Q}}
\newcommand{\cR}{{\mathcal R}}
\newcommand{\cS}{{\mathcal S}}
\newcommand{\im}{\mathrm{Im}}

\newcommand{\PGammaL}{\mathrm{P\Gamma L}}
\newcommand{\Sp}{\text{Sp}}
\newcommand{\wt}{\text{wt}}
\newcommand{\eps}{\varepsilon}
\newcommand{\<}{\langle}
\renewcommand{\>}{\rangle} 
\renewcommand{\phi}{\varphi} 
\newcommand{\gauss}[2]{\genfrac{[}{]}{0pt}{}{#1}{#2}}

\title{Degree 2 Boolean Functions on Grassmann Graphs}
\author{

Jan De Beule\footnotemark[1],

Jozefien D'haeseleer\footnotemark[2],\\

Ferdinand Ihringer\footnotemark[2],

Jonathan Mannaert\footnotemark[1]
}
\date{17 Oct 2022}

\begin{document}
\maketitle
\footnotetext[1]{Department of Mathematics and Data Science, research group Digital Mathematics, 
Vrije Universiteit Brussel, B--1050 Brussels, Belgium, 
{\tt \href{mailto:Jan.De.Beule@vub.be}{\{Jan.De.Beule,}\href{mailto:Jonathan.Mannaert@vub.be}{Jonathan.Mannaert\}@vub.be}}}
\footnotetext[2]{
Dept.~of Mathematics:~Analysis, Logic and Discrete
Math., Ghent University, Belgium,
\url{Jozefien.Dhaeseleer@ugent.be}, \url{Ferdinand.Ihringer@gmail.com}}

\begin{abstract}
 We investigate the existence of Boolean degree $d$ functions 
 on the Grassmann graph of $k$-spaces in the vector space $\Ff_q^n$.
 For $d=1$ several non-existence and classification results are known,
 and no non-trivial examples are known for $n \geq 5$. 
 This paper focusses on providing a list of examples on the case $d=2$ 
 in general dimension and in particular for $(n, k)=(6,3)$ and $(n,k) = (8, 4)$.
 
 We also discuss connections to  
 the analysis of Boolean functions, 
 regular sets/equitable bipartitions/perfect 2-colorings in graphs, 
 $q$-analogs of designs, and permutation groups. In particular, this
 represents a natural generalization of Cameron-Liebler
 line classes.
\end{abstract}

\section{Introduction}

The research presented here is motivated by a variety of open problems
in only loosely related areas such as finite geometry,
Boolean function analysis, association schemes and design theory.
Since it seems reasonable to assume that most readers
are not familiar with concepts and conventions
in all of these areas, we provide a relatively
long introduction. We refer to \cite{FI2019}
for a more detailed discussion of degree $1$ functions.

\smallskip 

More technical details and definitions are omitted
from the the introduction and can be found in 
Section \ref{sec:prelim}.

\subsection{Low Degree Boolean Functions} 

It is a well-known fact that one can write any $0,1$-valued (Boolean)
function on the hypercube $\{0, 1\}^n$ as a real, multilinear
polynomial of degree at most $n$. The study of such functions
which we can write as a polynomial of some bounded degree $d$
has been very fruitful. For instance, it has been observed
countless times that a Boolean degree one function on the hypercube
is of the form $0$, $1$, $x_i$, or $1-x_i$ for 
some $i \in \{ 1, \ldots, n\}$:

\begin{Theorem}[Folklore]
    A Boolean degree $1$ function on the hypercube depends on
    at most one coordinate.
\end{Theorem}

One of the fundamental results in {\it Boolean function analysis}
(see \cite{ODonnell2014} for a detailed introduction) is a 
characterization by Nisan and Szegedy of Boolean degree $d$ functions
in \cite{NiS1994}.

\begin{Theorem}[Nisan, Szegedy (1994)] \label{thm:NiS}
  A Boolean degree $d$ function on the hypercube depends on
  at most $d 2^{d-1}$ coordinates.
\end{Theorem}

Let $\gamma(d)$ denote the optimal upper bound for given $d$, that is
there exists a Boolean degree $d$ function depending on $\gamma(d)$
coordinates, but not depending on $\gamma(d)+1$ coordinates.
Nisan and Szegedy showed that $\gamma(d) \leq d 2^{d-1}$.
They also described a Boolean degree $d$ function with 
$2^{d}-1$ relevant variables.
Recently, better upper and lower bounds of magnitude 
$O(2^d)$ were found, see \cite{CHS2020} and subsequent work.
The better lower bound was, in a different context, 
first observed in \cite{CT2002}.

In the last few years, there has been interest in comparable results
on domains different from the hypercube. For instance see \cite{DFLLY2020}
or \cite{FI2019} and the references therein. 
One example is the {\it Johnson graph} $J(n, k)$, also
known as {\it slice} of the hypercube, which consists of all 
$k$-subsets of $\{ 1, \ldots, n \}$, two subsets adjacent when 
their intersection has size $k-1$. For instance, 
a classification of Boolean degree $1$ functions on the Johnson
graph has been obtained several times independently,
see \cite{Filmus2016} and \cite{Meyerowitz1992}. Note that
$k$ and $n-k$ have to be at least $2$ as otherwise all functions
have degree $1$.

\begin{Theorem}
    Let $n-k, k \geq 2$.
    A Boolean degree $1$ function on the Johnson graph $J(n, k)$ depends on
    at most one coordinate.
\end{Theorem}

Filmus and the third author generalized the result by Nisan and Szegedy
to the Johnson graph \cite{FI2019a}.

\begin{Theorem}[Filmus et al. (2019)] \label{thm:NiS_slice}
    There exists a constant $C$ such that the following holds. 
    If $C^d \leq k \leq n-C^d$ and 
    $f: \binom{\{ 1,\ldots, n\}}{k} \rightarrow \{ 0, 1\}$ has degree $d$,
    then $f$ depends on at most $\gamma(d)$ coordinates.
\end{Theorem}

As before, bounds on $k$ are necessary here, but $C^d \leq k \leq n-C^d$ seems overly
generous. 
Our (very) limited investigation here suggests 
that for $d=2$, we only need to exclude $(n, k) =(6, 3)$ 
and $k < 2d$, see Section \ref{sec:johnson}.
In fact, recently Filmus showed in \cite{Filmus2022} that if we do not 
insist on the upper bound $\gamma(d)$ derived from the hypercube, 
but just some upper bound in $O(1)$, then $2d \leq k \leq n-2d$ suffices.

Recently, Theorem \ref{thm:NiS_slice} has been extended to several other structures,
for instance the multislice \cite{FODW2019} by Filmus, O'Donnell and Wu, 
and to the perfect matching scheme by Dafni, Filmus, Lifshitz, Lindzey, and Vinayls \cite{DFLLY2020}.

\subsection{Cameron-Liebler Line Classes and 
Boolean Functions on the Grassmann Graph}\label{sec:intro_CL}

Our main focus are low degree Boolean functions on the {\it Grassmann graph} $J_q(n, k)$
which consists of all $k$-subspaces of an $n$-dimensional vector space over
the finite field of order $q$, two subspaces adjacent when their meet
has dimension $k-1$. Let $\cF$ be a family of $k$-spaces of $V \coloneqq \Ff_q^n$.
We read $\cF$ as a Boolean function over the reals, 
that is we identify it with the function $f$ from all $k$-spaces of $V$ to the reals,
where $f(S) = 1$ if $S \in \cF$ and $f(S) = 0$ otherwise.
Let $T$ be a subspace of $V$.
Let $x_T$ denote the family of all $k$-spaces
which are incident with $T$.
We say that a (not necessarily: Boolean) function $f$ 
has {\it degree $d$} if we can write
$f$ as a linear combination (over the reals) of all
$x_T$ with $\dim(T) = d$.

The study of Boolean degree $1$ functions, limited to $k=2$,
under the name of {\it Cameron-Liebler line classes} is actually older than 
most of the aforementioned results. In the Grassmann graph, $1$-spaces 
(in projective notation: {\it points})
or, equivalently, $(n-1)$-spaces ({\it hyperplanes}) are a natural choice for variables,
see \S\ref{sec:notation_BFA} for details. It was conjectured by Cameron
and Liebler in \cite{CL1982} that, as for the hypercube and the Johnson graph,
all degree $1$ examples are the trivial ones:

\begin{Conjecture}[Cameron, Liebler (1982)]\label{conj:CL}
    Let $n \geq 4$ and $k=2$.
    If $f$ is a Boolean degree $1$ function on the Grassmann graph $J_q(n, k)$,
    then $f$ depends on at most one point and one hyperplane.
\end{Conjecture}

More explicitly, the conjecture suggests that $f$ is one of $0, x_P, x_H, x_P+x_H, 1, 1-x_P,
1-x_H, 1 -x_P-x_H$ for some $1$-space $P$ and some
$(n-1)$-space $H$ with $P \nsubseteq H$. In the terminology
of finite geometry, they suggested that an example either consists
of none of the lines ($2$-spaces), a point-pencil, a dual point-pencil,
the union of a point-pencil and a dual point-pencil or the complement
of any of these examples. In a breakthrough result, Drudge showed in 
his PhD thesis \cite{Drudge1998} that this conjecture fails for $(n,k,q) = (4,2,3)$.
Nowadays many counterexamples to the Conjecture of Cameron and Liebler
are known if $(n,k) = (4, 2)$ and $q \geq 3$; 
see \cite{BD1999,DDMR2016,FMRXZ2021,FMX2015,GMP2016,GP2005}.
The general case of $k>2$ has been investigated more recently,
for instance see \cite{BDBD2019,DBM2021,DBMS2021,FI2019,Metsch2017,RSV2018}. 
Indeed, it has been shown in \cite{FI2019}
that Conjecture \ref{conj:CL} holds for $k \geq 2$ and $q \in \{ 2,3,4,5\}$
if we exclude the case $(n,k) = (4,2)$.

In analogy to Theorem \ref{thm:NiS_slice},
it seems natural to assume that when $n-k$ and $k$ are large enough,
that is $n-k, k \geq C(d)$ for some $C(d)$ independent of $n$ and $k$,
then all degree $d$ functions on the Grassmann graph $J_q(n, k)$ 
only depend on very few coordinates $x_P$ and $x_H$, that is on
the intersection with very few points and hyperplanes. Our 
results here give an indication of what $C(d)$ could be.

In Section \ref{sec:glob_pl} and in Section \ref{sec:glob_sd}
we construct Boolean degree $2$ functions which 
depend on many coordinates. For this, we use 
finite symplectic and orthogonal geometries.

\subsection{Equitable Bipartitions}

A {\it regular set} or {\it equitable bipartition} 
of a $k$-regular graph $\Gamma$ is a subset $S$ of vertices of $\Gamma$
such that there exists constants $a$ and $b$ such that a vertex in $S$ has 
precisely $a$ neighbors in $S$ and such that a vertex not in $S$ has precisely $b$ neighbors in $S$. 
The eigenvalues of the {\it quotient matrix}
\[
 \begin{pmatrix}
  a & k-a \\
  b & k-b
 \end{pmatrix}
\]
are also eigenvalues of the adjacency matrix of $\Gamma$ \cite[Lemma 9.3.1]{AlgComb}.
Equitable bipartitions are known under various other names,
for instance {\it perfect $2$-colorings}, {\it completely regular codes},
or {\it intriguing sets}, see also \cite{FI2019} and the references therein.

Boolean degree $1$ functions in some classical 
lattices, in particular Johnson and Grassmann graphs,
are equitable bipartitions (for instance, this follows
from our discussion in \S\ref{sec:eqdef}). 
More generally, equitable bipartitions
correspond to Boolean degree $d$ functions which (in the terminology 
of Boolean function analysis) have no weights on degrees in $\{ 1, \ldots, d-1\}$. Regular sets on the hypercube
are well-investigated, primarly due to Fon-der-Flaass \cite{FonDerFlaass2007}. 
There has been recent work
on the Johnson graph, most notably the 
equitable bipartitions of the Johnson graph $J(n, 3)$ have been classified
for $n$ odd, see \cite{GG2013}.
Recently, this attracted much research:
regular sets of degree $2$ have been classified in the Johnson graph 
by Vorob'ev in \cite{Vorobev2020};
Metsch and De Winter investigated small 
equitable bipartitions in the the Grassmann graph 
of planes $J_q(n, 3)$ \cite{DWM2020};
Mogilnykh surveyed equitable bipartitions 
in $J_2(6, 3)$ and $J_q(n, 2)$ \cite{Mogilnykh2020}.

\section{Preliminaries} \label{sec:prelim}

\subsection{Projective Geometry}

Using projective notation, in $\Ff_q^n$ we denote $1$-space 
as {\it points}, $2$-spaces as {\it lines}, $3$-spaces as {\it planes},
$4$-spaces as {\it solids}, 
$(n-2)$-spaces as {\it colines}, 
and $(n-1)$-spaces as {\it hyperplanes}.
For a vector space $V$, let $\gauss{V}{k}$ denote its $k$-dimensional subspaces.
We denote the $q$-binomial (or: Gaussian) coefficient 
by $\gauss{n}{k}_q \coloneqq | \gauss{V}{k} |$ for $V = \Ff_q^n$.
We write $[n]_q \coloneqq \gauss{n}{1}_q = \frac{q^n-1}{q-1}$. Usually, we do not put the $q$
and write $\gauss{n}{k}$ and $[n]$.
Note that
\[
 \gauss{n}{k} = \prod_{i=0}^{k-1} \frac{[n-i]}{[k-i]}.
\]

Several of our constructions use well-known finite simple groups,
namely the symplectic group $\Sp(n, q)$ 
and orthogonal group $O^\eps(n, q)$.
We try to keep this reasonably self-contained by providing explicit 
equations, but we refer to standard literature on finite geometry
and the classical finite simple groups for more details, see \cite{Taylor1992}.

\subsection{Analysis of Boolean Functions} \label{sec:notation_BFA}

Recall the concept of a Booleean degree $d$ function in $\Ff_q^n$ 
on $k$-spaces from \S\ref{sec:intro_CL}.
We say that a Boolean function $f$ is a {\it $j$-junta}
if there is a set $J$ of points and hyperplanes with
$|J| = j$ such that we can write $f$ as a polynomial
in $x_R$, $R \in J$.
A rather trivial example of a $2$-junta is the set $f$ of 
$k$-spaces inside a fixed hyperplane $\pi$ or trough a 
fixed point $p\not \in \pi$. As we can write 
$f = x_p + \alpha \sum_{r \in \gauss{\pi}{1}} x_r + \beta \sum_{r \in \gauss{V}{1} \setminus \gauss{\pi}{1}} x_r$ for suitable real
constants $\alpha$ and $\beta$, this is a Boolean 
degree $1$ function. It only 
depends on $x_p$ and $x_\pi$, thus it is a $2$-junta.

\bigskip 

We summarize easy, well-known facts in the following lemma.
It shows that for fixed degree $d$, induction on $n$ and $k$ is feasible.
Therefore it motivates our study of small $n$ and $k$ to establish 
an inductive basis.

\begin{Lemma}
  Let $n-d \geq k \geq d \geq 1$.
  Let $f, g$ be Boolean degree $d$ functions on $J_q(n,k)$ with $d \geq 1$.
  Let $P$ be a $1$-space and $H$ a hyperplane of $V \coloneqq \Ff_q^n$.
  Then all of the following have degree $d$:
  \begin{enumerate}[(a)]
   \item The (not necessarily Boolean) functions $0$, $1$, $f+g$, $f-g$, and $1-f$.
   \item The set $\{ S \in \gauss{H}{k}: f(S) = 1 \}$.
   \item The set $\{ S/P \in \gauss{V/P}{k-1}: f(S) = 1 \}$.
  \end{enumerate}
\end{Lemma}
\begin{proof}
    Clearly, any linear combination of two degree $d$ functions has degree $d$.
    This shows (a).
    
    For (b): Write $f$ as 
    \[
    f = \sum_{T \in \gauss{V}{d}} c_T x_T.
    \]
    For a $k$-space $S \subseteq H$, we have that $f(S) = 1$ if and only if $h(S) = 1$, where
    \[
     h \coloneqq \sum_{T \in \gauss{H}{d}} c_T x_T.
    \]
    So $h$ has degree $d$ and is the characteristic function of the set.
    
    The statements (b) and (c) are dual.
\end{proof}

\bigskip 

We will show in Lemma \ref{lem:up_bnd} that 
there exists a relatively easy upper bound
which shows that in the Grassmann graph any
function depends on at most $C_{n,k} q^{n-k}$
coordinates.

\begin{Lemma}\label{lem:meet_in_spanning}
    Let $n \geq 2k \geq 2$.
    Then there exists a $q_0$ such that for all $q \geq q_0$
    there exists a family $\cH$ of $(n-k+1)$-spaces in $\Ff_q^n$
    with $|\cH| = k^2(n-k+1)$
    such that for each $k$-space $S$ we have that
    $\< S \cap \bigcup_{H \in \cH} H \> = S$.
\end{Lemma}
\begin{proof}
    For a $k$-space $S$ and $(n-k+1)$-spaces 
    $H_1, \ldots, H_m$ with $m \leq k^2(n-k)$
    put $T_m(S) = \< S \cap \bigcup_{i=1}^m H_i \>$.
    Consider the property (P) that $\dim(T_m(S)) < m$.
    First we calculate the probability $p_m$ that
    random $H_1, \ldots, H_m$ have property (P).
    We claim that $ p_m < (1+o(1)) q^{m-k-1}$ (as $q \rightarrow \infty$) and 
    our proof proceeds by induction on $m$.
    
    Clearly, $p_1=0$ as $H_1$ intersects $S$ nontrivially.
    For $m>1$, the probability that $H_1, \ldots, H_{m-1}$
    satisfy (P) is $p_{m-1}$. 
    Suppose that $\dim(T_{m-1}(S)) \geq m-1$.
    If $\dim(T_{m-1}(S)) > m-1$, then $\dim(T_m(S)) \geq m$,
    so property (P) is not satisfied.
    If $\dim(T_{m-1}(S)) = m-1$, then
    there are $[k]-[m-1] = q^{m-1}[k-m+1]$ points in $S \setminus T_{m-1}(S)$
    and $[m-1]$ points in $T_{m-1}(S)$.
    Hence, the probability that $H_m$ meets $S$ only in $T_{m-1}(S)$
    is at most $\frac{[m-1]}{[k]}$ (as $H_m$ meets $S$ nontrivially
    and this is the probability of a point of $S$ being in $T_{m-1}(S)$). Hence, by the union bound for the two cases $\dim(T_{m-1}(S)) < m-1$ and $\dim(T_{m-1}(S)) \geq m-1$,
    \begin{align*}
     p_m < p_{m-1} + \frac{[m-1]}{[k]} < (1+o(1)) q^{m-k-1}.
    \end{align*}
    This completes the proof of the claim.

    Now let us pick $k^2(n-k+1)$ random $(n-k+1)$-spaces.
    Let $X$ denote the random variable which counts
    the number of $k$-spaces $S$ with $\dim(T_d(S)) < k$.
    Recall that $\gauss{n}{k} < (1+o(1)) q^{k(n-k)}$.
    Then
    \begin{align*}
      \Ef(X) = \gauss{n}{k} \cdot p_k^{k(n-k+1)} 
             < (1+o(1)) q^{k(n-k)} \cdot q^{-k(n-k+1)} < 1.
    \end{align*}
    Hence, by linearity of expectation,
    there exists a choice of $k^2(n-k+1)$ $(n-k+1)$-spaces
    with $T_m(S) = S$ for all $k$-spaces $S$.
\end{proof}

\begin{Lemma} \label{lem:up_bnd}
  Let $1 < k < n-1$. Then there exists a
  $q_0$ such that for all $q \geq q_0$ we have that
  any Boolean function on $J_q(n, k)$
  is a $k^2(n-k+1) \frac{q^{n-k+1}-1}{q-1}$-junta.
\end{Lemma}
\begin{proof}
  By duality, we assume that $k \leq n/2$.
  Put $V = \Ff_q^n$.
  By Lemma \ref{lem:meet_in_spanning}, we can find a set $\cH$ of
  $k^2(n-k+1)$ subspaces of dimension $n-k+1$ such that 
  any $k$-space $T$ contains at least $k$ points in 
  $\bigcup_{H \in \cH} H$ which span $T$. For each $T$, let us 
  denote $k$ such points by $\cP(T)$.
  Then
  \[
    f = \sum_{f(T)=1} \prod_{P \in \cP(T)} x_P. 
  \]
  Hence, we see that $f$ depends on at most 
  $k^2(n-k+1) \cdot \frac{q^{n-k+1}-1}{q-1}$ points.
\end{proof}


\subsection{The Spectra of Johnson and Grassmann Graphs}

Let $n \geq 2k$.
Consider the eigenvalues of the adjacency matrices 
of the Johnson graph $J(n, k)$ and the Grassmann
graph $J_q(n, k)$. These are well-understood objects,
for instance see Chapter 9 in \cite{BCN}.
Both graphs have $k+1$ eigenspaces $V_0, V_1, \ldots, V_{k}$
with a natural ordering: the eigenspace $V_d$ has
dimension $\gauss{n}{d}-\gauss{n}{d-1}$ (where we read $\gauss{n}{d} = \binom{n}{d}$
for the Johnson graph). Corollary 3.2.3 in \cite{Vanhove2011} implies

\begin{Lemma} \label{lem:span_eigenspaces}
Let $n\geq 2k$ and consider $J_q(n,k)$, then we have for every $0 \leq d\leq k$ that
$V_0 + \ldots + V_{d}=\< x_D: \dim D=d \>$.
\end{Lemma}

In particular, the first eigenspace
is spanned by the all-ones vector (or, as we identify vectors
and functions here, $f=1$).
Note that the eigenvalue of $J_q(n, k)$ corresponding to the eigenspace $V_j$ is
\[
  q^{j+1} [k-j][n-k-j] - [j].
\]

Asking for an equitable bipartitions
is the same as asking for a set with characteristic 
function in $V_0 + V_d$ for some $d$, see \cite[\S9.3]{AlgComb}.

\subsection{Some Equivalent Definitions} \label{sec:eqdef}

The next result emphasizes the (well-known) fact that 
there are three ways of looking at degree $d$ functions:
We can see them as degree $d$ polynomials, we can see them
as functions in certain eigenspaces, or we can see them 
as a certain type of functions in posets.
We include some minor variants which we consider useful.
We assume that $d \leq k$.

The {\it $d$-space-to-$k$-space incidence matrix} $A = (a_{ij})$ of $J_q(n, k)$
is the $(\gauss{n}{d} \times \gauss{n}{k})$-matrix indexed
by $d$-spaces and $k$-spaces of $\Ff_q^n$ where $a_{ij} = 1$
if the $i$-th $d$-space lies in the $j$-th $k$-space and $a_{ij} = 0$
otherwise. 

\begin{Proposition}
    Let $n \geq 2k$.
    For $f$ a real function on $J_q(n, k)$ the following are equivalent:
    \begin{enumerate}[(a)]
     \item The function $f$ has degree $d$.
     \item The function $f$ lies in $V_0 + \ldots + V_{d}$.
     \item The function $f$ is orthogonal to $V_{d+1} + \ldots + V_{n}$.
     \item There exists a weighting $\wt: \gauss{V}{d} \rightarrow \Rf$ such that for all $S \in \gauss{V}{k}$ we have
            \[ f(S) = \sum_{D \in \gauss{S}{d}} \wt(D). \]
     \item The function $f$ lies in the image of the $d$-space-to-$k$-space incidence matrix.
    \end{enumerate}
\end{Proposition}
\begin{proof}
Lemma \ref{lem:span_eigenspaces} shows the equivalence
of (a) and (b).
The equivalence of (b) and (c) follows 
from the fact that the common eigenspaces of the 
association scheme $J_q(n, k)$ are pairwise orthogonal
(as its adjacency matrices are symmetric).
Further, (a) and (d) are equivalent: 
If (a) holds, then we can write $f$ as
\[
    f = \sum_{D \in \gauss{V}{d}} c_D x_D.
\]
Take $\wt(D) = c_D$ to obtain (d).
Conversely, if (d) holds, then take $c_D = \wt(D)$
to see that $f$ has degree $d$. Note that  $D \in \gauss{V}{d}$
lies on some $S \in \gauss{V}{k}$.
Let $A$ denote the $d$-space-to-$k$-space incidence matrix.
Then (d) states that $f = A^T \cdot \wt$. 
Here we see $\wt$ as a vector of weights.
Hence, (d) and (e) are equivalent.
\end{proof}


\subsection{Boolean Functions and Designs}
Classical designs live in the Johnson graph.
Let $n \geq 2k \geq 2d \geq 0$.
A (classical) $d$-$(n, k, \lambda)$ design in the Johnson graph 
$J(n, k)$ is a family $\cD$ of $k$-sets such that each $d$-set 
lies in exactly $\lambda$ elements of $\cD$.
A $d$-$(n, k, \lambda)$ design in the Grassmann graph $J_q(n, k)$ is
a family $\cD$ of $k$-spaces such that each $d$-space lies in exactly 
$\lambda$ elements of $\cD$. The existence of these $q$-analogs of 
classical designs was settled (at least in some weak sense)
asymptotically by Fazeli, Lovett, and Vardy in 2014 \cite{FLV2014},
but for small parameters deciding existence is notoriosly hard.
Maybe most prominently, a classical $2$-$(7, 3, 1)$ design 
is well-known as the Fano plane. The existence of 
a $2$-$(7,3,1)$ design in $J_q(7, 3)$, the so-called $q$-analog,
is a long-standing open problem.

Let $\cD$ be a $d$-$(n, k, \lambda)$ design of the Grassmann graph $J_q(n, k)$.
By a standard double counting argument,
$|\cD| = \lambda \gauss{n}{d}/\gauss{k}{d}$.
For any family $\cF$ such that the characteristic function
$f$ of $\cF$ has degree $d$, then $|\cD \cap \cF|$ only depends
on $|\cD|$ and $|\cF|$. Indeed, Boolean degree $d$ functions
are precisely the objects with this property. It is a case
of what Delsarte called {\it design-orthogonality}, see also \cite{Delsarte1977}.

\begin{Corollary}\label{cor:div}
    Let $n \geq 2k$.
    Consider a $d$-$(n, k, \lambda)$ design $\cD$ of $J_q(n, k)$ 
    with characteristic function $g$.
    If $\cF$ is a degree $d$ subset of $J_q(n,k)$,
    then $|\cF \cap \cD| = |\cF| \cdot |\cD|/\gauss{n}{k}$.
    If also $\< g^\gamma: \gamma \in \PGammaL(n, q)\> = V_0 + V_{d+1} + \ldots + V_k$,
    then the converse holds too. 
\end{Corollary}

Hence, if $\cF$ has degree $d$, then $|\cD| \cdot |\cF|/\gauss{n}{k}$ is an integer.
This is a well-known generalization
of the fact that if $k$ divides $n$, then 
the size of a Boolean degree $1$ function is divisible 
by $\gauss{n-1}{k-1}$ (using $1$-$(n, k, \lambda)$ designs,
that is spreads).
We list the divisibility conditions which derive from the 
known designs in \S\ref{sec:div_cond}.

%
%

\section{Degree 2 in Hypercube and Johnson Scheme} \label{sec:johnson}

Consider the hypercube $\{ 0, 1\}^n$.
Nisan and Szegedy (Theorem \ref{thm:NiS}) showed that a Boolean degree $d$ function
on the hypercube depends on at most $d 2^{d-1}$ variables, so
a Boolean degree $2$ functions depends on at most 4 variables.
Hence, one can obtain a complete list
by considering the first four input variable $x,y,z,w$:
Up to permutation and negation of the input, Boolean degree 2 functions
are 
\begin{align*}
  &0, ~~x, ~~x \text{ AND } y = xy, ~~x \text{ XOR } y = x+y-xy, ~~xy+(1-x)z, \\
  &\text{Ind}(x{=}y{=}z) = xy+xz+yx-x-y-z+1, \\
  & \text{Ind}(x {\leq} y {\leq} z {\leq} w \text{ OR } x {\geq} y {\geq} z {\geq} w).
\end{align*}
Here $\text{Ind}(B)$ is the indicator of $B$.
This list was first obtained by Camion, Carlet, Charpin, 
and Sendrier in \cite{CCCS2001}. For $d=3$
a Boolean degree $2$ function depends on 
at most $10$ variables, see \cite{Zverev2007}.
Note that $3 \cdot 2^2 = 12 > 10$.

\bigskip 

Now consider the Johnson graph $J(n, k)$ with $n \geq 2k$.
For $k=3$ there are countless examples for degree $2$ functions which 
depends on an arbitrary amount of coordinates, see \cite{Filmus2022,FI2019a}
for more details and conjectures.
All equitable bipartitions of degree $2$ 
are classified for $k=3$ \cite{EGGV2022,GG2013}.
For $n$ divisible by $d$ and $k \leq 2d-1$, 
we can find a partition $\cL$ of $\{ 1, \ldots, n \}$ into $d$-sets.
Then
\[
    f(x) = \sum_{ S \in \cL} \prod_{i \in S} x_i
\]
has degree $d$. This function corresponds to the family of $k$-sets containing 
one of the $d$-sets of $\mathcal{L}$. This is a very special case of what Martin 
calls {\it groupwise complete design}, see \cite{Martin1994}.

For $J(8, 4)$ we found an example that depends on $5$
variables. Identify $J(8, 4)$ as a subset of $\{0,1\}^8$
and take all vertices which start with one of
\begin{align*}
  & 11000, 01100, 00110, 00011, 10001,\\
  & 11100, 01110, 00111, 10011, 11001.
\end{align*}
For an alternative description, 
let $Z$ be the cyclic group of order $5$
with its natural action on $\{ 1, \ldots, 5 \}$.
Then we can take any $4$-set which intersects
$\{ 1, \ldots, 5\}$ in one of the orbits $\{1,2\}^Z$
or $\{ 1,2,3\}^Z$.
It is an equitable bipartition with quotient matrix
\begin{align*}
 \begin{pmatrix}
   8 & 8 \\
   6 & 10
 \end{pmatrix}.
\end{align*}

Recall that Boolean degree $2$ function on the hypercube 
depends on at most $4$ coordinates. Thus, the behavior
of the Johnson graph is notably different from the hypercube.

\section{Examples for General Degree}

\subsection{Trivial Examples}

Let $2d \leq 2k \leq n$ and let $\perp$
be some polarity of $\Ff_q^n$.
For a $d$-space $T$, 
let $x_{T,i}$ denote all $k$-spaces $S$
with $\dim(S \cap T) = d-i$,
and $x_{T^\perp,i}$ denote all $k$-spaces $S$
with $\dim(S \cap T^\perp) = d-i$.
We call these examples {\it trivial}. We 
also call all examples trivial which one can 
obtain from these examples by taking unions,
differences, and complements.

For our main interest, $d=2$, there are three
examples to emphasize.

\begin{Example}
    \begin{enumerate}[(a)]
     \item The set of all $k$-spaces through a fixed $2$-space $L$: $x_L = x_P x_Q$.
        Here $P$ and $Q$ are points which span $L$.
     \item The set of all $k$-spaces in a fixed $(n-2)$-space $C$: $x_C = x_H x_K$.
     Here $H$ and $K$ are hyperplanes which intersect in $C$.
     \item The set of all $k$-spaces through a fixed $1$-space $P$ in a fixed $(n-1)$-space $H$:
        $x_P x_H$. Here $P \subseteq H$.
    \end{enumerate}
\end{Example}

Note that the last example is particularly interesting.
Let $C$ be an $(n-2)$-space. Let $\cH$ be the set of $q+1$
hyperplanes through $C$. For each hyperplane $H \in \cH$,
pick a point $P_H \subseteq H$ outside of $C$.
Then
\[
  f = \sum_{H \in \cH} x_H x_{P_H}
\]
is a Boolean degree $2$ function of size $(q+1) \gauss{n-1}{k-1}$.
It is a $2(q+1)$-junta.

%

\subsection{A (Partial) Spread} \label{sec:partial_spread}

Here we describe a union of trivial examples
which we consider noteworthy.
Let $\cS$ be a family of $d$-spaces of $\Ff_q^n$ which are pairwise disjoint.
Such a family is called a {\it partial spread}.
Clearly, $|\cS| \leq [n]/[d]$. In case of equality $\cS$ is called a {\it spread}.
Indeed, spreads exist if and only if $d$ divides $n$, see \cite{Beutelspacher1975}. 
The maximal size of $\cS$ when $d$ does not divide $n$ was determined recently
in \cite{NaS2017}. Clearly,
\[
    f = \sum_{S \in \cS} x_S
\]
is a Boolean degree $d$ function for $k$-spaces if $k \leq 2d-1$.
It shows that any type of Nisan-Szegedy theorem 
(for which we assume $q$ fixed and $n \rightarrow \infty$) needs to exclude 
the case $k \leq 2d-1$.

\subsection{Free Constructions from the Hypercube}

Let $h: \{ 0, 1\}^{m} \rightarrow \{ 0, 1 \}$ be a Boolean degree 
$d$ function on the hypercube. Further, take 
a linear independent set
$B = \{ b_1, \ldots, b_{m} \}$ in $\Ff_q^n$ 
(here $n \geq m$). 
Define a Boolean degree $d$
function $f$ on the subspaces $S$ of $\Ff_q^n$ by putting
\[
  f(S) = h((x_{\<b_i\>}(S))_{i \in \{ 1,\ldots, m\}}).
\]
In \cite{CHS2020}
a Boolean degree $d$ function is described which depends
on $m = \ell(d) \coloneqq 3 \cdot 2^{d-1}-2$ variables. Hence,
a Boolean degree $d$ function on 
$J_q(n, k)$ can depend on $\ell(d)$ variables.
If $B$ is not linear independent,
then $f$ depends on less than $\ell(d)$ variables.

\smallskip 

For $q$ fixed and $d$ sufficiently large,
this is the best construction for low degree Boolean
functions on $J_q(n, k)$ which we are aware of.

For $(n,k)=(8,4)$ we describe a Boolean degree $2$
function in \S\ref{sec:glob_sd} which seems to depend
on more than $C(q^4+q^3+q^2+q+1)$ variables
(we can only show that it depends on at least 
$q^3+q^2+q+1$ variables).

\section{Global Degree 2 Examples from Polar Spaces} 

The most famous example for a nontrivial Boolean degree $1$ function exists 
in $J_q(4, 2)$ (for $q$ odd) and is closely related to the elliptic quadric $O^-(4, q)$.
For degree $2$ we went through all polar spaces in small dimensions.

\subsection{Examples for Planes} \label{sec:glob_pl}

We consider examples on planes, that is $k=3$.

\subsubsection{Symplectic Spaces} \label{sec:symp}
Let $n \geq 6$.
  Consider a (possibly degenerate) symplectic form
  $\sigma$ on $\Ff_q^n$.
  If $n$ is even, then $\sigma(x, y) = x_1 y_2 - x_2 y_1 + \ldots + x_{n-1}y_n - x_n y_{n-1}$
  is a nondegenerate choice for $\sigma$. We say that $x, y$ are orthogonal
  if $\sigma(x, y) = 0$. Let $S$ be a subspace. We write $S^\perp$ for the 
  subspace of vectors orthogonal to $S$. The radical of $S$ is $S \cap S^\perp$.
  We say that $S$ is isotropic if its radical is $S$, and that $S$
  is nondegenerate if its radical is trivial.
  
  \smallskip 
  There are two types of $2$-spaces with respect to $\sigma$:
  Let $\cL_1$ denote the set of isotropic $2$-space,
  and let $\cL_2$ denote the set of nonisotropic $2$-spaces.
  
  \smallskip
  There are also two types of $3$-spaces with respect to $\sigma$:
  Let $\Pi_1$ denote the set of isotropic planes, and
  let $\Pi_2$ denote the set of planes with a point as a radical.
  
  \smallskip 
  
  We claim that $\Pi_i$ has degree $2$ for $i \in \{ 1, 2\}$:
  Put
  \[
   f = \frac{1}{q^2+q+1} \sum_{L \in \cL_1} x_L 
   - \frac{q+1}{q^2(q^2+q+1)} \sum_{L \in \cL_2} x_L.
  \]
  Clearly, $f$ has degree $2$ and corresponds to the set $\Pi_1$. 
  It remains to see that $f$ is Boolean.
  All $q^2+q+1$ lines in an
  isotropic plane $\Pi$ are isotropic, so $f(\Pi)=1$.
  A plane $\Pi$ with a point as radical has $q+1$ isotropic lines
  and $q^2$ nondegenerate lines, so then $f(\Pi)=0$.
  
  \medskip 
  
  Now assume that $n$ is even and that $\sigma$ is nondegenerate.
  
  \smallskip
  
  The symmetry group of this example is $\Sp(n, q)$.
  The number of line orbits equals the number of plane orbits
  as there are precisely two types of each.
  Thus, this is an example for equality in Block's lemma, Lemma \ref{lem:block}.
  The group $\Sp(n, q)$ acts transitive on points, hence the example
  is a $1$-design and therefore an equitable bipartition. For $n=6$, this
  was already observed in \cite{DWM2020}.
  The quotient matrix is
  \begin{align*}
   \begin{pmatrix}
     q[3][n-5] & q^{n-4} [2][3] \\
     [2][n-4] & [3]q[n-3] - [2][n-4]
   \end{pmatrix}.
  \end{align*}
%
%
  
  \begin{Remark}
   A similar construction works for even degree $d$ 
   and $k=d+1$.
  \end{Remark}

\subsubsection{Quadrics} \label{sec:quadrics}

Let $n \geq 6$.
Consider a quadratic form $Q$ on $\Ff_q^n$,
for instance $Q(x) = x_1^2 + x_2x_3 + \ldots +x_{n-1}x_n$
for $n$ odd.
Let $\cQ$ denote the singular points $\<x\>$ (so $Q(x)=0$).
We say that a subspace of $\Ff_q^n$ is totally singular 
if all its points are singular.

Let $\cL_i$ denote the family of lines which intersect $\cQ$
in $i$ points. As $Q$ is quadratic, $\cL_i$ is empty unless 
$i \in \{ 0, 1, 2, q+1 \}$. A line in one of these sets is called exterior line,
tangent, secant, or totally singular line, respectively.
There are five types of planes with respect to $Q$.
In bracket we provide the explicit isomorphy type in $\Ff_q^3$.

\smallskip
\noindent
Let $\Pi_1$ denote the family of totally singular planes (isomorphic to the quadratic form $Q'(x) = 0$).

\smallskip
\noindent
Let $\Pi_2$ denote the family of planes of double line type (isomorphic to $Q'(x)=x_1^2$).

\smallskip
\noindent
Let $\Pi_3$ denote the family of planes with exactly one singular point 
(isomorphic to $Q'(x)=x_1^2 + \alpha x_1 x_2 + \beta x_2^2$ such that 
$x_1^2 + \alpha x_1 x_2 + \beta x_2^2$ is irreducible over $\Ff_q$).

\smallskip
\noindent
Let $\Pi_4$ denote the family of planes with 
exactly two totally singular lines (isomorphic to $Q'(x)=x_1x_2$).

\smallskip
\noindent
Let $\Pi_5$ denote the family of conic planes (isomorphic to $Q'(x) = x_1^2 + x_1x_2$).

\smallskip 

Let $A = (a_{ji})$ denote the $5 \times 4$ matrix such that $a_{ji}$
denotes the number of lines of $\cL_i$ in a plane of $\Pi_j$.
Then 
\[
 A = \begin{pmatrix}
    0 & 0 & 0 & q^2+q+1 \\
    0 & q^2+q & 0 & 1\\
    q^2 & q+1 & 0 & 0\\
    0 & q-1 & q^2 & 2\\
    \binom{q}{2} & q+1 & \binom{q+1}{2} & 0
 \end{pmatrix}.
\]
Then \[ A \left( -\frac{q+1}{q^4+q^3+q^2}, \frac{1}{q^2+q+1}, 
-\frac{q+1}{q^4+q^3+q^2}, \frac{1}{q^2+q+1}\right)^T = ( 1, 1, 0, 0, 0)^T,\]
so $\Pi_1 \cup \Pi_2$ has degree $2$.
We can write the characteristic function of it as
\[
    f_1 = -\frac{q+1}{q^4+q^3+q^2} \sum_{L \in \cL_0 \cup \cL_2} x_L
        + \frac{1}{q^2+q+1} \sum_{L \in \cL_1 \cup \cL_{q+1}} x_L.
\]
If $n$ and $q$ are even, and $Q$ is of hyperbolic type, 
then this example is isomorphic to the symplectic example 
in Section \ref{sec:symp}, but not when $q$ is odd.
Its quotient matrix is identical to the symplectic example.

For $q=2$, we also find the following example which corresponds to $\Pi_1 \cup \Pi_3$:
\[
  f_2 = \frac{15}{64} \sum_{L \in \cL_0} x_L
  -\frac{1}{42} \sum_{L \in \cL_1} x_L
  -\frac{11}{168} \sum_{L \in \cL_2} x_L
  +\frac{1}{7} \sum_{L \in \cL_{q+1}} x_L.
\]
For $(n,q)=(6,2)$ with $Q$ of elliptic type $O^-(6, q)$, $\Pi_1$ is empty. 
Hence, for $q=2$ the sets $\Pi_i= \Pi_1 \cup \Pi_i$ for $i=2,3$ 
have degree $2$, and so any $\Pi_i$ with $i \in \{ 2, \ldots, 4 \}$ 
has degree 2.

\subsection{Examples for Solids} \label{sec:glob_sd}

We consider examples on solids, that is $k=4$.
Let $n=8$.
Let $Q$ be a nondegenerate quadratic form of elliptic type $O^-(8, q)$,
for instance $Q(x) = x_1^2 + \alpha x_1 x_2 + \beta x_2^2 
+ x_3^2 + \ldots + x_8^2$ such that $x_1^2 + \alpha x_1 x_2 + \beta x_2^2$
is irreducible over $\Ff_q$.
The terminology is identical to \S\ref{sec:quadrics},
so $\cQ$ is the set of singular points and we partition 
the lines set into $\cL_0 \cup \cL_1 \cup \cL_2 \cup \cL_{q+1}$.
There are the following types of solids.
In bracket we provide the explicit isomorphy type in $\Ff_q^4$.

\smallskip 
\noindent
Let $\cS_1$ denote the set of all solids of double plane type 
(with a quadratic form of type $Q'(x) = x_1^2$).

\smallskip 
\noindent
Let $\cS_2$ denote the set of all solids with two totally singular planes 
(type $Q'(x) = x_1x_2$).

\smallskip 
\noindent
Let $\cS_3$ denote the set of all solids with with precisely one totally singular line 
(type $Q'(x) = x_1^2 + \alpha x_1x_2 + \beta x_2^2$).

\smallskip 
\noindent
Let $\cS_4$ denote the set of all solids that intersect $\cQ$ in a cone
with a point as base over a conic (type $Q'(x) = x_1^2 + x_1x_2$).

\smallskip 
\noindent
Let $\cS_5$ denote the set of all nondegenerate solids of hyperbolic type $O^+(4, q)$
(type $Q'(x) = x_1x_2 + x_3x_4$).

\smallskip 
\noindent
Let $\cS_6$ denote the set of all nondegenerate solids of elliptic type $O^-(4, q)$
(type $Q'(x) = x_1^2 + \alpha x_1 x_2 + \beta x_2^2 + x_3^2$).

\smallskip 

The example below can be seen as a generalization of the example by 
Bruen and Drudge for degree $1$ in \cite{BD1999}.
Let $A = (a_{ji})$ denote the $6 \times 4$ matrix such that $a_{ji}$
denotes the number of lines of $\cL_i$ in a solid of $\Pi_j$.
\begin{align*}
A=	\begin{pmatrix}
		0 & q^2 (q^2+q+1) & 0 & q^2+q+1 \\
		0 & q(q^2-1) & q^4 & 2q^2+2q+1 \\
		q^4 & q(q+1)^2 & 0 & 1\\
		\frac12 q^3(q-1) & q^3+2q^2 & \frac12 q^3(q+1) & q+1 \\
		\frac12 q^2(q-1)^2 & (q+1)(q^2-1) & \frac12 q^2(q+1)^2 & 2(q+1) \\
		\frac12 q^2(q^2+1) & (q+1)(q^2+1) & \frac12 q^2(q^2+1) & 0
	\end{pmatrix}.
\end{align*}
Then we see that
\[
  A \left( \frac{q+1}{q^3(q^2+q+1)}, 0, - \frac{q+1}{q^3(q^2+q+1)}, \frac{1}{q^2+q+1} \right)^T
  = (1,1,1,0,0,0)^T.
\]
Hence, $\cS_1 \cup \cS_2 \cup \cS_3$ is a degree $2$ set. 
The corresponding degree $2$ polynomial is
\[
 f = \frac{q+1}{q^3[3]} \left((\sum_{L \in \cL_0} x_L) - (\sum_{L \in \cL_2} x_L)\right)
  + \frac{1}{[3]} \sum_{L \in \cL_{q+1}} x_L.
\]
Note that $|\cS_1  \cup \cS_2 \cup \cS_3| = (q^4+1)(q^3+1)(q^2+1) [5]$.

\smallskip 
This example for $J_q(8,4)$ seems to depend on almost all coordinates
which is in contrast to $J(8,4)$ where we only obtained an example
depending on $5$ coordinates, see Section \ref{sec:johnson}. 
Formally, we can show the following:

\begin{Proposition}\label{prop:rel_vars}
  The example $\cS_1  \cup \cS_2 \cup \cS_3$ depends on at least $q^4-q^3+q^2-q+3$ variables
  of type $x_P$ and $x_H$ for $P$ a $1$-space and $H$ an $(n-1)$-space.
\end{Proposition}
\begin{proof}
  Put $\cS = \cS_1 \cup \cS_2 \cup \cS_3$.
  On average, a point of $\Ff_q^8$ lies in $(q^3+1)(q^2+1)[5]$
  elements of $\cS$. A singular point lies in 
  $(q+1)(q^2+1)^2(q^3+1)$ elements of $\cS$.
  Hence, a non-singular point lies on 
  \begin{align*}
    &\frac{[8]\cdot (q^3+1)(q^2+1)[5] - (q^4+1)(q^2+q+1)\cdot (q+1)(q^2+1)^2(q^3+1) }{[8] - (q^4+1)(q^2+q+1)}\\
    &= (q^3+1)[4]^2 =: \alpha
  \end{align*}
  elements of $\cS$.
  Dually, a hyperplane of $\Ff_q^8$
  contains at most $\alpha$ elements of $\cS$.
  Hence, we need at least 
  \[
    \left\lceil \frac{|\cS|}{\alpha} \right\rceil = q^4-q^3+q^2-q+3
  \]
  points and hyperplanes to cover all elements of $\cS$.
  Now suppose that $\cS$ only depends on a set $\cR$ of less than $q^4-q^3+q^2-q+3$
  points and hyperplanes. 
  Then there exist $4$-spaces 
  $S \in \cS$ and $T \not\in \cS$ which are non-incident
  with all elements of $\cR$. Hence, we cannot distinguish 
  between $S$ and $T$ based on $\cR$ which contradicts
  that $\cS$ only depends on $\cR$.
\end{proof}

By Lemma \ref{lem:up_bnd}, the preceding result 
is tight up to a constant factor (as $q \rightarrow \infty$).
Hence, there exists a degree $2$ function on $J_q(8, 4)$
which depends, up to a constant factor, the maximum number
of coordinates.

\section{Other Examples}

\subsection{Local Degree 2 Examples for Planes}

Here we provide an (incomplete) selection of 
examples which are of degree $2$ and which stabilize
a partial flag of subspaces, but are not trivial.

\subsubsection{A Line and a Complementary Spread}

For $n=6$, let $L$ be a line and let $\cC$ be a set of $q^2+1$
colines through $L$ which pairwise meet in $L$. Note that $\cC$
exists because $\Ff_q^4$ possesses line spreads.
Then the set
$\{ \Pi \text{ a plane}: \dim(\Pi \cap L) = 0 \text{ and } 
\exists C \in \cC: \Pi \subseteq C \}$ has degree $2$
and size $(q^2+1) \cdot q^2(q+1)$.
We can write its characteristic function $f$ as
\[ f=\sum_{C \in \cC} (x_{C}-x_L). \]

The example is a $(q+1)(q^2+2)$-junta: 
we can decide if an element is in the set by testing inclusion 
for each of the $q+1$ hyperplanes through the $q^2+1$ 
colines through $L$, together with testing the inclusion 
for each of the $q+1$ points of $L$.

\subsubsection{Incident Point-Plane-Hyperplane} \label{sec:S82}

\begin{figure}
 \centering
 \includegraphics[scale=0.26]{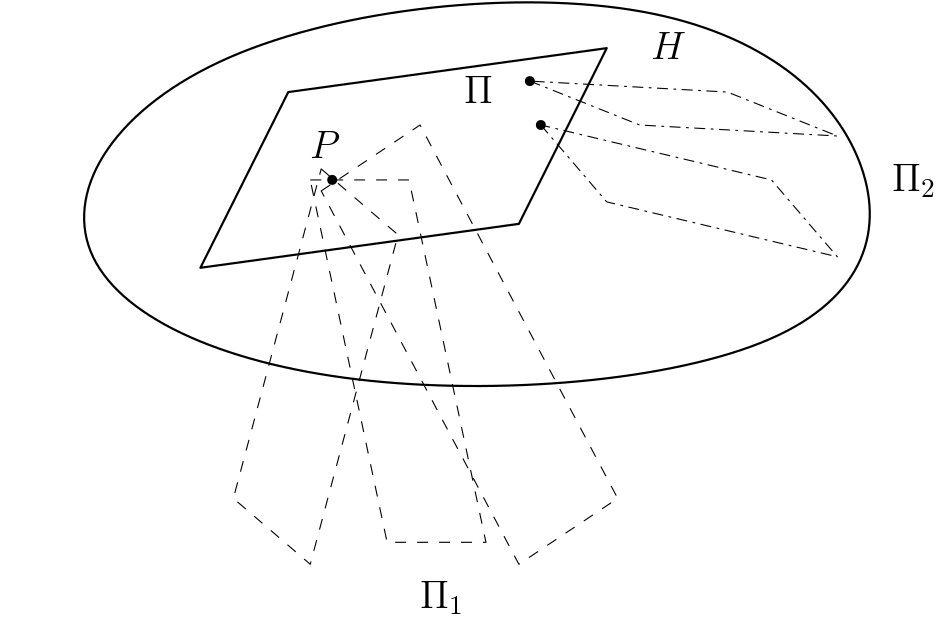}
 \caption{The point-plane-hyperplane example from \S\ref{sec:S82}.
 The planes of $\Pi_1$ and $\Pi_2$ correspond to the planes with dashed border.}
 \label{fig:S82ex}
\end{figure}

Let $n=6$. Pick a point $P$, a plane $\Pi$, and a hyperplane $H$
such that $P \subseteq \Pi \subseteq H$.
 
 \smallskip
 
 \noindent
 Let $\Pi_1$ be the set of all planes not in $H$ which meet $\Pi$
 in a line through $P$.

 \smallskip
 
 \noindent
 Let $\Pi_2$ the set of all planes in $H$ whose meet with $\Pi$ 
 is a point different from $P$.
 
 \smallskip
 
Then $\Pi_1 \cup \Pi_2$ corresponds to a degree $2$ function
and has size $q^3(q+1) + (q^2+q)q^4 = (q^2+1) \cdot q^3 (q+1)$.
To see that it has degree 2, consider the types of lines:

\smallskip
\noindent
Let $\cL_1$ be the set of all lines in $\Pi$ through $P$.

\smallskip
\noindent
Let $\cL_2$ be the set of all lines in $\Pi$ not through $P$.

\smallskip
\noindent
Let $\cL_3$ be the set of all lines in $H$ whose meet with $\Pi$ is $P$.

\smallskip
\noindent
Let $\cL_4$ be the set of all lines in $H$ whose meet with $\Pi$ is a point, but not $P$.

\smallskip
\noindent
Let $\cL_5$ be the set of all lines in $H$ which are skew to $\Pi$.

\smallskip
\noindent
Let $\cL_6$ be the set of all lines whose meet with $H$ is $P$.

\smallskip
\noindent
Let $\cL_7$ be the set of all lines whose meet with $H$ is a point in $\Pi$, but not $P$.

\smallskip
\noindent
Let $\cL_8$ be the set of all lines whose meet with $H$ is a point not in $\Pi$.

\smallskip

Then we can write the characteristic function $f$ of $\Pi_1 \cup \Pi_2$ as
\begin{align*}
f = &\frac{q^3}{[3][2]} \sum_{L \in \cL_1} x_L
+ \frac{-q}{[3]} \sum_{\cL_2 \cup \cL_3} x_L\\
&+ \frac{1}{[3][2]} \sum_{\cL_4 \cup \cL_7} x_L
+ \frac{q+1}{q[3]} \sum_{\cL_5 \cup \cL_6} x_L
+ \frac{-1}{q^2[3]} \sum_{\cL_8} x_L.
\end{align*}
See Figure \ref{fig:S82ex} for an illustration.
The example is a $(q^2+q+2)$-junta: we can decide
if an element is in $\Pi_1 \cup \Pi_2$ by testing
inclusion for the $q^2+q+1$ points in $\Pi$ and $H$.

\subsection{Some Sporadic Examples}

Here are some sporadic example for $(n,q) = (6,2)$.
Despite our best efforts, we did not manage to generalize them.
The reader can find some more sporadic examples, which are also 
equitable bipartitions, in \cite{Mogilnykh2020}.

\subsubsection{Incident Line-Solid}

For $n=6$ and $q=2$, let $M$ be a line and let $C$ be a coline with $M \subseteq C$.
There are $q+1$ hyperplanes through $C$ and $q+1$ points on $M$.
For each of the $q+1$ points $P \subseteq M$, choose a 
distinct hyperplane $H_P$ through $C$.

\smallskip
\noindent
Let $\Pi_1$ denote the set of all planes $\pi$ with $M \subseteq \pi \subseteq C$.

\smallskip
\noindent
Let $\Pi_2$ denote the set of all planes $\pi$ not in $C$ 
which meet $M$ in some point $P$ and satisfy $\pi \subseteq H_p$.

\smallskip
\noindent
Let $\Pi_3$ denote the set of all planes which meet $M$ in some point
 $p$ and and $C$ in a line. Note that $\Pi_2 \subseteq \Pi_3$.
 
 \smallskip
 
The set $\Pi_1 \cup \Pi_3$ has degree $2$ and size 
$(q+1) + q^3(q+1)^2 + q^3(q^2+1)(q+1)$.
It is a $2(q+1)$-junta: test all points on $L$ and all hyperplanes through $C$.

Then set $\Pi_1 \cup \Pi_2$ has degree $2$ and size $(q+1) + q^3(q+1)^2$.
The example is a $2(q+1)$-junta (as before). Its characteristic function can be written as
\[ f = \sum_{P \in \gauss{M}{1}} (x_P - x_P x_{H_P}),\]
where $f$ is the characteristic function of $\Pi_1 \cup \Pi_2 \cup \Pi_3$ and also has degree $2$.

\subsubsection{Incident Point-Line-Plane-Hyperplane}

For $n=6$ and $q=2$, let $P$ be a $1$-space,
$M$ a $2$-space, $\Pi$ a $3$-space, and 
$H$ a $5$-space such that 
$P \subseteq M \subseteq \Pi \subseteq H$.

\smallskip
\noindent
Let $\Pi_1$ denote all planes in $H$ which contain $M$.

\smallskip
\noindent
Let $\Pi_2$ denote all planes not in $H$ 
which meet $\Pi$ in a line through $P$ different from $M$.

\smallskip
\noindent
Let $\Pi_3$ denote all planes in $H$ which meet $\Pi$
in a point on $M$ different from $P$.

\smallskip 
The set $\Pi_1 \cup \Pi_2 \cup \Pi_3$ has degree $2$ and 
size $7+16+32 = 55$.
The example is a $(3q+1)$-junta, that is a $7$-junta:
we can decide
if an element is in $\Pi_1 \cup \Pi_2$ by testing 
inclusion for $H$ and the points in a triangle which includes $M$.

\subsection{Unexplained Computer Examples}

We found some examples by computer which we could not derive 
from any of the other examples. We present these and their 
symmetries in the following table. 
In the structure description, we denote the the cyclic group of 
order $m$ by $C_m$, the symmetric group of order $m!$ by $S_m$,
and the dihedral group of order $m$ by $D_m$.
We write $a^b$ if an orbit of length $a$ occurs $b$ times.
For plane orbits, we only provide those that constitute the degree $2$ example.

{\bigskip\footnotesize\noindent
\setlength{\tabcolsep}{3pt}
\begin{tabular}{c|cllll}
 size & stab size & point orbits & line orbits & plane orbits & structure\\
\hline
 80 & 61440
    & $32, 20, 10, 1$ 
    &$320, 160, 60,  40^2, 16, 10, 5$ 
    & $40^2$ 
    & $C_2^5 \rtimes (  C_2^4 \rtimes S_5 )$ \\  
 85 & 86016
    &$ 32, 14^2, 2, 1$ 
    &$224^2,  84, 32, 28^2, 16, 7^2, 1$ 
    & $56, 21, 8$ 
    & $G$, see below\\
177 & 64
    & $32, 4^7, 2^4, 1^7$ 
    &$32^8, 16^{15}, 8^8, 4^8, 2^{24}, 1^{11}$ 
    & $16^4, 8^{10}, 2^{13}, 1^7$ 
    & $C_2^3 \times D_8$ \\
420 & 126
    & $42, 21$ 
    & $126^4, 63, 42, 21, 14, 7$ 
    & $126^3,42$
    & $S_3 \times ( C_7 \rtimes C_3)$\\
\end{tabular}

\smallskip \noindent
Here $G = (C_4{\times} C_2^3) {\rtimes} ( C_2 {\times} ( C_2^3 {\rtimes} PSL(3,2)))$.
\par}

\section{Concluding Remarks}

(1) In case of the Johnson graph, 
let us remark that a classification of Boolean degree $2$ functions in $J(n, k)$ appears feasible,
but goes beyond the scope of the present work.

\medskip 
\noindent
(2) For $d=1$ and $k=2$, Bamberg and Penttila \cite{BP2008} classified all 
subgroups of $\PGammaL(n, q)$ which have the same number of orbits on points and lines.
This answered a conjecture by Cameron and Liebler in \cite{CL1982}.
In light of \S\ref{sec:block}, the following generalization is natural:
\begin{Problem}
  Classify all subgroups of $\PGammaL(n, q)$ with the same number of orbits
  on $d$-spaces and $k$-spaces, that is all Boolean degree $d$ functions
  on $J_q(n, k)$ for which equality holds in Lemma \ref{lem:block}.
\end{Problem}

\medskip 
\noindent
(3) In \cite{DFLLY2020} it was shown for several domains that one can 
write a Boolean degree $d$ function as a constant depth decision tree.
In case of the Grassmann graph, the natural queries are 
``Is a point contained in a subspace?'' and 
``Is a subspace in a hyperplane?''. 
In light of our collection of examples, one is tempted to make the following conjecture:
\begin{Conjecture}\label{conj:filmus}
    For every given $q$ and $d$, there exists a $k_0$ such that if $k, n-k \geq k_0$,
    then every Boolean degree $d$ function on $J_q(n, k)$ is a constant depth decision tree.
\end{Conjecture}

One might also conjecture that the depth only depends polynomially on $q$.
A related, but less specific conjecture can be found in \cite{FI2019a}.

\bigskip
\paragraph*{Acknowledgment} We thank Sam Adriaensen, John Bamberg, 
and Alexander L. Gavrilyuk for very helpful comments.
We thank Yuval Filmus for several suggestions,
including the list in \S\ref{sec:johnson} and Conjecture \ref{conj:filmus}.
We thank Yuriy Tarannikov 
for informing us about \cite{CCCS2001} and 
\cite{Zverev2007}.
We thank the referee for their careful reading of the document.
The second and third authors are each supported by a 
postdoctoral fellowship of the Research Foundation -- Flanders (FWO).

\appendix

\section{Permutation Groups and Block's Lemma} \label{sec:block}

In 1982 Cameron and Liebler investigated subgroups of $\PGammaL(n, q)$
and their orbits on points and lines of the vector space of the vector space $\Ff_q^{n}$, see \cite{CL1982}.  
An application of Block's lemma \cite{Block1967} shows that any subgroup of $\PGammaL(n, q)$ 
has at least as many orbits on lines as it has on points. In case of equality the orbits of lines have, in 
our terminology here, degree $1$.

In this section we concisely discuss a generalization of this application of
Block's lemma to degree $d$, i.e. we show that a subgroup of $\PGammaL(n, q)$ 
has at least as many orbits on $k$-spaces of $\Ff_q^{n}$ as on $d$-spaces if $d \leq k \leq n/2$.
Again, if equality occurs, then the orbits on $k$-spaces have degree $d$ in the Grassmann 
graph $J_q(n, k)$.

\begin{Lemma}[\cite{Block1967}]\label{lem:block}
Let $G$ be a group acting on two finite sets $X$ and $X'$, with respective sizes $n$ and $m$.
Let $O_1,\ldots,O_s$, respectively $O'_1,\ldots,O'_t$ be the orbits of the action on $X$, respectively $X'$.
Suppose that $R \subseteq X \times X'$ is a $G$-invariant relation and call $A = (a_{ij})$ the $n \times m$ matrix of this relation, i.e. $a_{ij}=1$ if and only if $x_i R x'_j$ and $a_{ij} = 0$ otherwise, after having ordered the 
elements of $X$ and $X'$ arbitrarily. Let $\chi_S$ denote the characteristic vector of a set $S$.
\begin{itemize}
\item[(i)] The vectors $A^T \chi_{O_i}$, $i=1, \ldots, s$, are linear combinations of the vectors $\chi_{O'_j}$.
\item[(ii)] If $A$ has full row rank, then $s \leq t$. If $s = t$, then all vectors $\chi_{O'_j}$ are linear combinations
of the vectors $A^T \chi_{O_i}$, hence $\chi_{O'_j} \in \im(A^T)$.
\end{itemize}
\end{Lemma}


Let $G \leqslant \PGammaL(n,q)$, $n \geq 4$, let $2 \leq k < n$, $d \leq k$ and $d \leq n/2$, and let $X$, 
respectively $X'$ be the set of $d$-spaces, respectively $k$-spaces of $\Ff_q^{n}$. 
The incidence, i.e. the symmetrised set theoretic containment, 
between an element of $X$ and $X'$ is $G$-invariant. Furthermore, the incidence matrix is the incidence matrix 
of the $k$-space design of $\Ff_q^{n}$, and this matrix has full row rank by \cite{Bose1949}. 
If $s=t$, i.e. if $G$ has equally many orbits on the $d$-spaces as on the $k$-spaces, then the characteristic 
vector of each of the orbits of $k$-spaces lies in $\im(A^T)$. 

In \cite{CL1982}, Cameron
and Liebler studied collineation groups having equally many point as line orbits, that is $(d,k) = (1,2)$.
Conjecture~\ref{conj:CL} translates in this context that such
a group is line transitive, or fixes a hyperplane and acts transitively on the lines of the hyperplanes, or, dually, fixes a point and acts transitively 
on the lines through the fixed point. 

We call two $k$-spaces of $\Ff_q^{n}$ {\em skew} if and only if they only share the zero vector. A partition 
of the points of $\Ff_q^{n}$ in $k$-spaces is called a {\em spread of $\Ff_q^{n}$ in $k$-spaces}. It is well known 
that such a spread exists if and only if $k \mid n$, e.g. when $n=4$ and $k=2$, there are spreads of lines in $\Ff_q^{n}$.
By Proposition 3.1 of \cite{CL1982}, for any set $L$ of lines of $\Ff_q^{4}$, $\chi_L \in \im(A^T) \iff |L \cap S| = x$ 
(a natural number, only depending on $L$), and for any line spread $S$ of $\Ff_q^{4}$.
Often, a Cameron-Liebler line class of $\Ff_q^{4}$ is defined using its characterization with relation to 
line spreads of $\Ff_q^{4}$. When $k \nmid n$, a set $K$ of $k$-spaces of $\Ff_q^{n}$ is
called a Cameron-Liebler set $k$-spaces of $\Ff_q^{n}$ if and only if $\chi_{K} \in \im(A^T)$. This is 
the case $d=1$ and $k \geq 1$ found in a geometrical context in \cite{BDBD2019}.

Note that the statements of Block's lemma are only unidirectional, i.e. an orbit of $k$-spaces under a collineation group with equally 
many point orbits as orbits on $k$-spaces is a Cameron-Liebler set of $k$-spaces, but the converse is not true, the union of all $k$-spaces through a fixed
point $P$ and contained in a fixed hyperplane not through $P$ is a Cameron-Liebler set of $k$-spaces which is not the orbit under a collineation group 
with equally many point as $k$-space orbits.

We do not know if a classification of such
subgroups of $\PGammaL(n, q)$ is feasible, 
but we will see in \S\ref{sec:symp}
that the symplectic group $\Sp(n, q)$ provides us with some examples
when $n$ and $d$ are even.

\section{Divisibility Conditions} \label{sec:div_cond}

In the following we summarize known divisibility conditions
based on the survey by
Braun, Kiermaier, Wassermann \cite{BKW2018}.

\subsection{Small Parameters}

For $(n,k) = (6,3)$, a $2$-$(6, 3, c(q+1))$ design has size $c(q^3+1)[5]$.
Existence is known for $(q,c) = (2,1), (3,3), (4,2), (5,13)$.
Hence, for $q=2,3,4,5$ we obtain that $|\cF|$
needs to be divisible by $5, 10, 17, 2$, respectively.

\smallskip 

For $(n,k)=(7,3)$, a $2$-$(7, 3, \lambda)$ design has size $\lambda (q^2-q+1) [7]$.
Existence is known for $(q, \lambda) = (2,3), (3,5), (4,21), (5,31)$.
Hence, for $q=2,3,4,5$ we obtain that $|\cF|$
needs to be divisible by $[5] = \frac{q^5-1}{q-1}$.

\smallskip 

For $(n,k)=(8,4)$, a $2$-$(8, 4, c(q^2+q+1))$ design has size $c (q^4+1)[7]$.
Existence is known for $(q,c) = (2,7),(3,455),(4,5733),(5,20181)$.
Hence, for $q=2,3,4,5$ we obtain that $|\cF|$
needs to be divisible by $93, 121, 341, 781$, respectively.

\subsection{Suzuki's construction}

Let $q$ be a prime and $n \geq 7$ be an integer
satisfying $\text{gcd}(n,4!)=1$.
Then there exists a $2$-$(n,3,q^2+q+1)$ design. See \cite[Theorem 11]{BKW2018}.
Hence,

\begin{Lemma}
  Let $\cF$ be a degree $2$ family of $3$-spaces in $\Ff_q^n$.
  Then $(q^3-1) |\cF|$ is divisible by $q^{n-2}-1$.
\end{Lemma}

For instance, for $n=11$ and $q=2,3,4,5$ or $7$: 
$(q^3-1)|\cF|$ is divisible by $511, 19\, 682, 262\,143, 1\, 953\, 124$ 
or $40\,353\,606$, respectively.

\subsection{More Conditions in the Binary Case}

\begin{Lemma}
Let $m\geq 3$. Suppose that $\cF$ is a set 
of $3$-spaces in $\Ff_2^n$ of degree 2, then the following holds:
\begin{enumerate}[(a)]
\item If $n=8m$, then $C|\cF|$ is divisible by $2^{8m-2}-1$, where $C\in \{42,312\}$.
\item If $n=9m$, then $42 \cdot |\cF|$ is divisible by $2^{9m-2}-1$.
\item If $n=10m$, then $210 \cdot |\cF|$ is divisible by $2^{10m-2}-1$.
\item If $n=13m$, then $42 \cdot |\cF|$ is divisible by $2^{13m-2}-1$.
\end{enumerate}
\end{Lemma}
\begin{proof}
By \cite[Section 5.2]{BKW2018}, there exist 
(a) $2$-$(8m,3, C)_2$ designs for $C \in \{ 42, 312\}$,
(b) $2$-$(9m,3, 42)_2$ designs for $m\geq 3$,
(c) $2$-$(10m,3, 210)_2$ designs for $m\geq 3$, and
(d) $2$-$(13m,3, 42)_2$ designs for $m\geq 3$.
\end{proof}

\end{document}